\documentclass[11pt]{amsart}
\usepackage{geometry, url}                
\usepackage{graphicx}
\usepackage{amssymb, amsmath}
\usepackage{epstopdf}
\newtheorem{theorem}{Theorem}[subsection]
\newtheorem*{theorem*}{Theorem}
\newtheorem{lemma}[theorem]{Lemma}
\newtheorem*{lemma*}{Lemma}
\newtheorem{lemma-def}[theorem]{Lemma-Definition}
\newtheorem{proposition}[theorem]{Proposition}
\newtheorem{corol}[theorem]{Corollary}

\theoremstyle{definition}
\newtheorem{definition}[theorem]{Definition}
\newtheorem*{definition*}{Definition}

\newcommand{\Q}{\mathbb{Q}}

\newcommand{\F}{\mathbb{F}}

\newcommand{\Z}{\mathbb{Z}}

\newcommand{\N}{\mathbb{N}}

\newcommand{\R}{\mathbb{R}}

\newcommand{\Lm}{\mathcal{L}}

\newcommand{\ord}{\mathrm{ord}\,}

\newcommand{\Lring}{\Lm_{\text{ring}}}

\newcommand{\Laff}{\Lm_{\text{aff}}}

\newcommand{\Ldist}{\Lm_{\text{dist}}}

\newcommand{\LmM}{\Lm_{M}}

\newcommand{\Lstar}{\Lm_*}

\newcommand{\ac}{\text{ac}\,}

\title{Cell Decomposition for Semibounded $p$-adic Sets}
\author{Eva Leenknegt}

\address{
              \email{eleenkne@math.purdue.edu} \\
          \url{http://www.math.purdue.edu/~eleenkne}    }   

\begin{document}
\maketitle
\begin{abstract}
We study a reduct $\Lstar$ of the ring language where multiplication is restricted to a neighbourhood of zero. The language is chosen such that for $p$-adically closed fields $K$, the $\Lstar$-definable subsets of $K$ coincide with the semi-algebraic subsets of $K$. Hence structures $(K, \Lstar)$ can be seen as the $p$-adic counterpart of the $o$-minimal structure of semibounded sets.

We show that in this language, $p$-adically closed fields admit cell decomposition, using cells similar to $p$-adic semi-algebraic cells. From this we can derive quantifier-elimination, and give a characterization of definable functions. In particular, we conclude that multiplication can only be defined on bounded sets, and we consider the existence of definable Skolem functions.
%
%
\end{abstract}
\section{Introduction}
We are interested in structures $(\F, \Lm)$ that satisfy the following minimality property (which we will call `$\Lring$-minimality'): the $\Lm$-definable subsets of $\F$ should coincide with the $\Lring$-definable subsets of $\F$. 
If $\F$ is a real closed field, all $\Lring$-definable subsets of $F$ are already $(<)$-definable, and hence such structures will be $o$-minimal.
When $\F$ is a $p$-adically closed field and $\Lm \supseteq \Lring$, we get $P$-minimal structures \cite{has-mac-97}. The language we consider in this paper is a reduct of $\Lring$, and hence not $P$-minimal, yet it is still $\Lring$-minimal. It is studied here as part of a larger project to describe such weak $p$-adic structures.

From now on we will assume that $K$, the universe of $(K,\Lm)$, is a $p$-adically closed field. Let us first introduce the languages that will be of interest. In \cite{clu-lee-2011}, we showed that any language satisfiying the above minimality property (for $p$-adic fields) has to be an extension of the \emph{minimal} language $\LmM$, consisting, for every $n,m \in \N\backslash\{0\}$, of relations \[R_{n,m}(x,y,z):= y-x \in zQ_{n,m}.\] 
  For $K = \Q_p$, the sets $Q_{n,m}$ are defined as $\cup_{k\in \N}p^{kn}(1+p^m\Z_p)$; we will give a more general definition in the next section. If we add function symbols for addition and scalar multiplication $\overline{c}: x \mapsto cx$ , we obtain  the semi-affine  \cite{lee-2011}  languages 
 \[ \Laff^F:= \LmM \cup (+, \{\overline{c}\}_{c\in F}),\] for fields $F \subseteq K$.
 (If $F = K$, we will also write $\Laff$ for $\Laff^K$.) 
 All fields where the symbols $Q_{n,m}$ can be defined admit cell decomposition in this language, see \cite{clu-lee-2011,lee-2011}. The definable functions have a very simple form: up to a finite partitioning in cells, the component functions are just linear polynomials (with coefficients in $F$ and constant term in $K$).
 
  In this paper we focus on languages $\Laff^F \cup \{*\}$, where $*$ is a function symbol for a restricted multiplication map
 \[*: (x,y) \mapsto g(x) \cdot g(y), \]
 with $g(x) = x$ if $\ord x =0$, and $g(x) = 1$ otherwise.
 In order to obtain quantifier elimination, we will  add a symbol `$|$' for the relation $\ord x < \ord y$. This relation is already definable in $\LmM$, but not necessarily in a quantifier-free way, see eg. \cite{clu-lee-2011}.
For $p$-adically closed  fields, the sets $P_N$ can be defined as finite unions of cosets $\lambda Q_{n,m}$ (and vice versa), and therefore the languages $\Laff^F \cup \{*\}$ are equivalent (when comparing their definable sets) to
\[\Lm_{*}^{F}:=(+,-,*, \overline{c}_{c\in F}, |,   \{P_n^*\}_{n>1}),\] 
where $P_n^*$ is the two-variable relation
\[P_n^*(x,y) \leftrightarrow y \in xP_n.\]\\
The main tool used to study the languages listed above is cell decomposition. 
Generally speaking, a cell is a definable set where the last variable $t$ has been `singled out': the relation between $t$ and the other variables $x_i$ is described using a formula $\phi(x,t)$ which has a fixed form for all cells. This fixed form often helps to simplify proofs. 

It is well-known that all definable sets of $o$-minimal structures can be partitioned as a finite union of cells. 
Unfortunately, there are no such `free rides' in the $p$-adic context. Let us give a brief overview of what is known.
Originally used to give an alternative proof of Macintyre's  quantifier elimination result \cite{mac-76}, the most influential result is probably Denef's cell decomposition \cite{denef-86} for $p$-adic semi-algebraic sets. It provided a blueprint for cell decomposition in the $P$-minimal and $\LmM$-minimal context. Roughly, $p$-adic cells for a structure $(K, \Lm)$ are $\Lm$-definable sets of the following form:
\[\{(x,t) \in D \times K \mid \ord a_1(x)\ \square_1\ \ord (t-c(x))\ \square_2\ \ord a_2(x), \ t-c(x) \in \lambda Q_{n,m}\}\]
with $D \subseteq K^l$ an $\Lm$-definable set, and $\square_i$ may denote `$<$' or no condition. The function $c(x)$, the \emph{center} of the cell should also be $\Lm$-definable. Of course one can replace the $Q_{n,m}$ by just $P_n$ if preferred. Usually it is also required that $a_1(x)$ and $a_2(x)$ should be $\Lm$-definable functions; we will call such cells \emph{strong} cells. 

 Mourgues \cite{mou-09} showed that $P$-minimal structures  admit cell decomposition using strong cells if and only if they have definable Skolem functions.
  This result can be extended, see \cite{lee-2011b}, to most $\LmM$-minimal structures. 
  However, checking that a structure has definable Skolem functions is usually only possible once the definable sets and functions are well-understood, which makes these results less useful for studying individual structures. Moreover, for weaker languages the existence of Skolem functions is far from certain.  The author's PhD Thesis \cite{phd} contains several examples of weak $p$-adic structures $(K, \Lm)$ 
   that do not admit definable Skolem functions (see the preprints  \cite{lee-2011b,lee-2012.1} on the author's home page for more details).
Note that this does not necessarily imply that cell decomposition is not possible, but rather that cell decomposition using strong cells is not possible. See \cite{lee-2011b} for more details.
\\\\
The first main result of this paper is that
\begin{theorem}
Let $K$ be a $p$-adically closed field, and $F \supseteq \Q$ a subfield of $K$. The structure $(K, \Lstar^{F})$ admits cell decomposition and elimination of quantifiers.
\end{theorem}
The proof is inspired by Denef's cell decomposition theorem for semi-algebraic sets. To ease the notation, we will only present the proof for $\Lstar:= \Lstar^K$. 
However, all results in sections 1-3 are valid (or can be adapted in a straightforward way) for structures $(K, \Lstar^{F})$ where $F$ may be any subfield of $K$.

As a next step, we give a characterization of definable functions in Section \ref{sec:deffun}.
\begin{theorem}
Let $X$ be a bounded semi-algebraic set. Let $f: D \subseteq K^n \to K$ be a function definable in $\Laff\, \cup\, \{X\}$. There exists a finite partioning of $D$ into $\Lstar^K$-definable sets, such that on each part $A$, 
\[ f_{|A} = p(x) + b(x),\]
where $p(x)$ is a linear polynomial in $K[x]$ and $b(x)$ is a bounded function.
\end{theorem}
\noindent Note that this implies that multiplication can at most be definable on  bounded sets.

 For structures $(K, \Lstar^F)$ we get a similar result, where now $p(x)$ is a linear polynomial with coefficients in $F$ and constant term in $K$. The existence of definable Skolem functions depends on the availabilty of scalar multiplication: we have Skolem functions on the condition that $F \subseteq \overline{\Q}^K$, where $\overline{\Q}^K$ is the algebraic closure of $Q$ in $K$. 
\\\\
It is interesting to compare these structures with their  $o$-minimal counterparts (see section \ref{sec:compare}). Peterzil \cite{pet-92} considered the structure $(\R, +,-, \cdot_{|[-1,1]}, \overline{c}_{c\in \R}, <)$, where multiplication is restricted to the interval $[-1,1]$.   He shows that it consists precisely of all semi-bounded semi-algebraic sets (that is, sets definable in $(+,-,<, \{X_i\}_{i\in I})$, where $\{X_i\}_{i\in I}$ is the collection of all bounded semi-algebraic sets). Moreover, it is the only proper substructure between real semi-algebraic sets and real semi-linear sets.

We will show that $(K,\Lstar^{K})$-definable sets  are exactly the $p$-adic semi-bounded semi-algebraic sets.
We also give some indications that, in contrast to the real case, the semi-bounded $p$-adic sets are not the only structure between the semi-affine and the semi-algebraic sets (we intend to explore this in further detail in a subsequent paper.)
\subsection{Preliminaries}
Let $K$ be a $p$-adically closed field; write $R_K$ for the valuation ring, $M_K$ for the maximal ideal of $R_K$ and $\Gamma_K$ for the value group. 
The sets $Q_{n,m}$ can be defined as follows. Fix an element $\pi$ with minimal positive valuation, and put 
\[Q_{n,m} :=\{ x \in P_n \cdot (1 + M_K^m) \mid \overline{\ac}_m(x) =1\},\]
where $\overline{ac}_m: K^{\times} \to (R_K \mod \pi^m)^{\times}$ is the unique group homomorphism such that ${\overline{ac}_m(\pi) =1}$ and $\overline{ac}_m(u) \equiv u \mod \pi^m$,  for every unit $u \in R_K$. Note that this is a natural generalization of the definition of $Q_{n,m}$ we gave in the introduction. For more details, we refer to \cite{clu-lee-2011}.
 The following notation will be used  frequently:
\[\rho_{n,m}(x) = \lambda \Leftrightarrow x \in \lambda Q_{n,m}.\]
Our eventual goal is to use cell decomposition to show that $p$-adically closed fields  admit quantifier elimination in the language $\Lstar$. A first step will be to ensure that the cells we use are definable without quantifiers. To achieve  this, we will restrict  to functions that are quantifier-free definable, in the following sense:
\begin{definition} Let $QFD_*^k$ be the collection of all functions $f: K^k \to K$ that satisfy the following: For any quantifier-free $\Lstar$-definable set $S \subseteq K \times K^r$, the set \[\{(x,y) \in K^{k+r} \ | \ (f(x),y)\in S\}\] is also $\Lstar$-definable without quantifiers. \end{definition}
\noindent $\Lstar$-cells can then be defined inductively as sets of the following form:
\begin{definition}
A subset of $K^{k+1}$ is called an $\Lstar$-cell if it is a set of the following form
\[\{(x,t) \in D\times K \ | \ord a_1(x) \ \square_1\  \ord(t-c(x))\ \square_2\ \ord a_2(x);\  t-c(x) \in \lambda P_n\},\]
with $a_i(x), c(x) \in QFD_*^k,\  \lambda \in K, D $ an $\Lm_*$-cell in $K^k$, and   $\square_i$ denotes `$<$' or `no condition'.  
\end{definition}
\noindent Our aim is to show that every definable set can be partitioned into a finite union of cells. It is not so hard to see that this is true for sets of this form:
\begin{lemma-def}
An $\Lstar$-precell in $K^{k+1}$ is a subset of $K^{k+1}$ that can be defined by a conjunction of conditions of the forms
\begin{eqnarray*}
\ord (a_1t +b_1(x)) &\square&\ord (a_2t +b_2(x)) ,\\
a_3t+b_3(x) &\in& \lambda Q_{n,m},
\end{eqnarray*}
with $\lambda \in K, a_i \in K; n, m \in\N,\ \square$ denotes $ <,\leqslant,=,\geqslant,$ or $>$, and $b_i(x) \in QFD_*^k$.\\
Any $\Lstar$-precell can be partitioned into a finite union of $\Lstar$-cells. 
\end{lemma-def}
\begin{proof}
See the proof of lemma 2.2 of \cite{clu-lee-2011}. (Or see \cite{lee-2011b}.)
\end{proof}

\subsection{First observations on definable functions}

Because of our restrictions on multiplication, we cannot assume that elements of  $K[x]$ are definable for all $x\in K$. Instead, we will be working with the following kinds of `polynomials':
\begin{definition}
Let $\mathcal{F}(r')$ be a set of functions $K^{r'} \to K$ in variables $x_1,\ldots,  x_{r'}$. 
\begin{enumerate}
\item Let $\psi(t_1,\ldots, t_r, y_1,\ldots, y_{l})$ be an $\Lm_*$-term containing variables $t_1,\ldots, t_r; y_1, \ldots, y_l$.
Choose functions $f_i(x) \in \mathcal{F}(r')$. Then
\[P(t,x) := \psi(t_1, \ldots, t_r, f_1(x), \ldots, f_l(x))\]
is called an $\Lm_{*}$-polynomial in variables $t_1,\ldots, t_r$ with parameters from $\mathcal{F}(r')$.\\
This collection of such $\Lm_*$-polynomials will be denoted by $K_{\mathcal{F}(r')}[t_1,\ldots, t_r]$.
\item The degree of a polynomial $P(t,x) \in K_{\mathcal{F}(r')}$ is the largest number $s \in \N$ for which there exist $a_1,\ldots, a_l \in K$, an open subset $D$ of $K^r$ and a polynomial $g \in K[t]$ of degree $s$, such that $\psi(t_1,\ldots, t_r, a_1, \ldots, a_l) = g(t)$
for all $t \in D$.
\end{enumerate}
\end{definition}
\noindent Remark: When we talk about `an $\Lm_*$-polynomial $f$ with parameters from $K$', this is meant in the sense of the above definition (with $\mathcal{F}(r') = \{f_a: x\mapsto a\ | \ a \in K\}$). In short notation: $f\in K_K[t_1,\ldots, t_r]$. 

We will show that in $\Lstar$, multiplication is only definable on bounded sets, that is: sets $X \subset K^n$ such that for a fixed $k \in \Z$,   $\min \ord x_i >k$ for all $(x_1, \ldots, x_n) \in X$.
 A function $f: X \to K$ will be called bounded  if there exists a bounded set $B$ such that $f(X) \subset B$.
\\\\
\noindent The following functions are  examples of functions in $QFD_*^2$:
\begin{lemma}
Let $X$ be a bounded subset of $K^2$. There exists an element of $K_K[x,y,z]$ that defines multiplication on $X$.
\end{lemma}
\begin{proof} 
If $\ord x, \ord y >0$, put $x\cdot y = (1+x) *(1+y)-x-y-1$.
More generally, if $\ord x, \ord y >-\gamma$ for some $\gamma>0$, choose $q_{\gamma}\in K$ with $\ord q_\gamma>\gamma$.  Now $\ord q_\gamma x >0$ and $\ord q_\gamma y >0$, so we can express $x\cdot y$ by writing $x\cdot y = \frac1{q_\gamma^2}(q_\gamma x \cdot q_\gamma y)$ and using the previous observation.
\end{proof}
\begin{lemma}\label{lemma:exqfd}
For every $\gamma \in \Z$, the following function is in $QFD_*^2$:
\[\div_{\gamma}: K^2 \to K: (x,y) \mapsto \left\{\begin{array}{ll} \frac{x}{y} & \min\{ \ord x, \ord y, \ord \frac{x}{y} \} > \gamma, \\0 & \text{otherwise.}\end{array}\right.
\]
\end{lemma}
\begin{proof}
Note that it is necessary to put a lower bound on the order of $\frac{x}{y}$, since we need the set
$\{(x,y,z)\in K^3 \mid z\cdot y =x\}$ to be definable. 

It can be checked (using Lemma \ref{lemma:lm*poly} below) that for any functions $f_1(x), f_2(x)\in QFD_*^k$ and any $\Lstar$-term $P(t,y)$, there is a finite partition of $K^{k+r}$ into precells (those are by definition quantifierfree definable sets),  such that on each precell $A$, there exist  functions $g_1(x,y) \in QFD_*^{k+r}$ and $g_2(y) \in K_K[y]$ such that for all $(x,y) \in A$, \[P(f_1(x) \div_k f_2(x),y) = (g_1(x,y) \div_k f_2(x)) +g_2(y).\] From this it follows easily that a set $\{(x_1, x_2,y)\in K^{2+k} \mid (x_1 \div_k x_2,y)\in S\}$ is quantifierfree definable whenever $S$ is. \end{proof}
\noindent Since $*$ is in general not distributive with respect to $+$, we could not  define an  $\Lstar$-polynomial to be simply a sum of monomials. However, up to a  finite partition into precells, it is possible to write $\Lstar$-polynomials in a more manageable way.

%



\begin{lemma}\label{lemma:lm*poly}
Let $g_r$ be an $\Lm_{*}$-polynomial $g_r \in K_{QFD_*^k}[t]$ of degree $r$ in one variable $t$, with coefficients that are functions of $x_1,\ldots, x_k$. 
There exists a finite partition of $K^{k+1}$ into precells $C$, such that on each $C$, $g_r$ can be written as $g_r = \sum_{j} g_{r,j}$, with
\begin{equation}g_{r,j}= \prod_{i=1}^{r_j}s_{ij}[a_{ij}(x)*(q_{ij}t +b_{ij}(x))]\cdot \prod_{i=1}^{r_{j'}}(v_{ij}t +c_{ij}(x)),\label{eq:lm*poly}\end{equation}
with $1\leqslant r_j+r_j' \leqslant r$. The coefficients $s_{ij}, q_{ij}$ and $v_{ij}$ are in $K$, and $a_{ij}(x), b_{ij}(x), c_{ij}(x)$ are functions from $QFD_*^k$. For $1\leqslant i \leqslant r_{j'}$,  the functions $v_{ij}t +c_{ij}(x)$ are bounded unless  $(r_j,r_{j'})=(0,1)$. 
In particular, there exist  $q\in K$ and $c(x) \in QFD_*^k$ and a bounded function $b(x,t) \in QFD_*^k$ such that 
\[g_r(x,t) = [qt + c(x)] + b(x,t).\]
\end{lemma}
\begin{proof}
 The proof is by induction on the degree of the $\Lm_*$- polynomial.
It is easy to see that any $\Lm_*$-polynomial in $K_{QFD_*^k}[t]$ of degree 1 can be written in the form
\[ \sum_i s_i[a_i(x)*(r_it+b_i(x))] +vt+c(x),\]
with $s_i,r_i,v \in K$, and $a_i(x),b_i(x), c(x) \in QFD_*^k$. This is clearly of the form we proposed.\\ Assume now that the lemma holds for $\Lm_*$-polynomials of degree $n$. 
Let $g_{n_0},g_{n_1}$ be two $\Lstar$-polynomials with respective degrees $n_0,n_1<n$. It is sufficient to show that if $n_0+n_1>n$, the lemma still holds for $g_{n_0}*g_{n_1}$. 
\\
Applying the lemma for $g_{n_0}$ and $g_{n_1}$, we find a partitioning of $K^{k+1}$ in precells $A$ such that on each cell $A$, $g_{n_k} = \sum_j g_{n_k,j}$, where each $g_{n_k,j}$ is of the form
\begin{equation*}g_{n_k,j}= \prod_{i=1}^{n_{k,j}}s_{ij}[a_{ij}(x)*(q_{ij}t +b_{ij}(x))]\cdot \prod_{i=1}^{n_{k,j'}}(v_{ij}t +c_{ij}(x)),\end{equation*}with $1 \leqslant n_{k,j} + n_{k,j'} \leqslant n_k$. The main point we need to check is that it is possible to partition $A$ in smaller precells on which  a  condition of the form
\begin{equation}\ord g_{n_k} \ \square \ 0, \label{eq:lstarzero}\end{equation}
(where $\square$ may denote `$<$', `$=$' or `$>$') holds for both $g_{n_0}$ and $g_{n_1}$. 

Let us first check that this is indeed true. 
Note that  $\ord s_{ij}[a_{ij}(x)*(q_{ij}t +b_{ij}(x))] = \ord s_{ij}$. Moreover, by the induction hypothesis we know that  unless $(n_{k,j}, n_{k,j'}) =(0,1)$, all functions $v_{ij}t +c_{ij}(x))$ are bounded, and hence $g_{n_k,j}$ is bounded, say with lower bound $l_{k,j}$.
The order of $g_{n_k}$ can only be smaller than $ l=\min_j l_{k,j}$ if  $g_{n_k}$ contains a linear term $vt+c(x)$ for which $\ord (vt+c(x)) < l$. 
Therefore the condition $\ord g_{n_k}<l$ is equivalent to $\ord(vt+c(x))<l$, which is a precell condition.\\
How does this help us to express that $\ord g_{n_k} =0$? By the reasoning above, we can express whether $\ord g_{n_k}<l$ or $\ord g_{n_k}\geqslant l$ using precell conditions. For a fixed integer $\kappa \in \Z$, the condition $\ord g_{n_k} = \kappa$ can then also be expressed: if $\kappa <l$, this is equivalent with $\ord (vt+c(x)) = \kappa$, and for $\kappa \geqslant l$ we can proceed as follows.
Require that $\ord g_{n_k}\geqslant l$. This implies that  each term $g_{n_k,j}$ is bounded, and that we can express that $g_{n_k,j} \equiv \alpha \mod \pi^{\kappa+1} R$ (for $\alpha \in K$ with $\ord \alpha \geqslant l$) using only conditions linear in $t$. Combining such conditions we can describe all possible cases where $\ord g_{n_k} = \kappa$. 

To conclude the proof, partition $K^{r+1}$ in precells $A$ such that conditions of type \eqref{eq:lstarzero} hold on each precell. The claims from the lemma can now be checked easily.
\end{proof}

\begin{proposition}\label{prop:lm*poly}
Let $f(x,t)$ be in $K_{QFD_*^k}[t]$. There exists a finite partition of $K^{k+1}$ in precells, such that on each precell $C$, one of the following situations occurs
\begin{enumerate}
\item  The function $f(x,t)$ is unbounded on $C$ and there exist $q\in K$ and $c(x) \in QFD_*^k$ such that  for all $(x,t)\in C$,
\begin{eqnarray*}
\ord f(x,t) &=& \ord (qt-c(x)),\\
\rho_{n,m}(f(x,t))&=& \rho_{n,m}(qt-c(x)),
\end{eqnarray*}
\item There exist a function $c(x) \in QFD_*^k$ and $r \in \N$,  such that for all $(x,t)\in C$:
\begin{equation*}f(x,t) = a_0(x) +a_1(x)(t-c(x))+ \ldots+ a_r(x)(t-c(x))^r,\end{equation*}
where $a_i(x) \in QFD_*^k$,  and $[t-c(x)]$ and the $a_i(x)$ are all bounded functions.
\end{enumerate}
\end{proposition}
\noindent Remark: If $f(x,t)$ is  in $K_K[x]$, then also $c(x) \in K_K[x]$.
\begin{proof}
By Lemma \ref{lemma:lm*poly}, there exists $l\in \Z$ such that 
\[f(x,t) = qt-c(x) + (\text{terms of order } \geqslant l).\]
This implies  that on $ \{(x,t) \in K^{k+1} \ | \ \ord (qt-c(x)) <l-m\}$, it holds that
\[\ord f(x,t) = \ord (qt-c(x)) \quad \text{and} \quad \rho_{n,m}(f(x,t)) = \rho_{n,m}(qt-c(x)).\]
Moreover, on $C = \{(x,t) \in K^{k+1} \ | \ \ord (qt-c(x)) \geqslant l-m\}$, if we write
\begin{equation*}f_{|C}(x,t)= \sum_j\left[\prod_{i=1}^{r_j}s_{ij}[a_{ij}(x)*(q_{ij}t -b_{ij}'(x))]\cdot \prod_{i=1}^{r_{j'}}(v_{ij}t -c_{ij}'(x))\right],\end{equation*}
then for all $i, j$, the functions $[q_{ij}t-b'_{ij}(x)]$, and $[v_{ij}t-c'_{ij}(x)]$ are bounded. If $q_{ij} \neq 0$, write $q_{ij}t -b_{ij}'(x) = q_{ij}(t-b_{ij}(x))$, and analogously for $v_{ij}$. Partition $C$ further in precells on which there is some linear term $t-c(x)$ such that
\[\ord (t-c(x)) = \min_{i,j} \{\ord (t-c_{ij}(x)), \ord (t-b_{ij}(x))\}.\]
We can then write 
\[t-c_{ij}(x) = [t-c(x)] +[c(x)-c_{ij}(x)],\]
where the order of $c(x)-c_{ij}(x)$ is bounded, since
\[\ord [c(x)-c_{ij}(x)] = \ord [(t-c_{ij}(x)) - (t-c(x))] \geqslant \ord (t-c(x)).\]
(And similarly for $c(x)-b_{ij}(x)$.)
By partitioning further if necessary, we may assume that either $\ord (t-b_{ij}(x)) \neq -\ord q_{ij}$ for all $(x,t) \in C$ (in which case we can replace the factor $s_{ij}(a_{ij}(x)*q_{ij}(t-b_{ij}(x)))$ by $s_{ij}(a_{ij}(x)*1))$, or $\ord (t-b_{ij}(x)) = -\ord q_{ij}$ for all $(x,t) \in C$. In this last case,
\[a_{ij}(x)*[q_{ij}(t-b_{ij}(x)] = q_{ij}(a_{ij}(x)*1)(t-b_{ij}(x)).\] We can use this to rewrite $f_{|C}(x,t)$ as
\begin{align*}
 \sum_j a_j(x)&\left[\prod_{i=1}^{\widetilde{r_j}}s_{ij}'[(a_{ij}(x)*1)(t -c(x))+ (a_{ij}(x)*1)(c(x)-b_{ij}(x))]\cdot \right.\\
&\  \left. \prod_{i=1}^{r_{j'}}v_{ij}[(t -c(x))+(c(x)-c_{ij}(x))]\right],\end{align*}
where the first product now contains all factors $(t-b_{ij})$ for which we know that $\ord t-b_{ij}(x) = -\ord q_{ij}$ on the given precell $C$.
Note that all factors occuring in this expression are bounded. This simplifies to a sum of terms of the form $d_{ij}(x)(t-c(x))^i$, where all $d_{ij}(x)$ are bounded functions belonging to $QFD_*^k$.
\end{proof}

\section{Cell decompostion} \label{sec:semi-bound}

%

\noindent We will show that $p$-adically closed fields $K$ admit cell decomposition in the language $\Lm_*$, using a method that is based on Denef's \cite{denef-86} proof of cell decomposition for semi-algebraic sets.
The following two propositions form the core of our proof:

\begin{proposition}[Cell Decomposition-preparation I]
Let $t$ be one variable and $x=(x_1,\ldots, x_k)$. Let $f(x,t)$ be in $K_{QFD_*^k}[t]$. There exists a finite partition of $K^{k+1}$ in cells $A$, such that on each cell, there is a center $c(x) \in QFD_*^k$ and $l \in \Z$ such that one of the following is true on each $A$. 
\begin{enumerate}
\item The function $f(x,t)$ is unbounded on $A$ and there exists  $q\in K$ such that  \[\ord f(x,t)= \ord (qt-c(x)).\]
\item The function $f(x,t)$ is bounded on $A$, and we can expand $f(x,t)$ as
\[ f(x,t) = a_0(x) +a_1(x)(t-c(x))+ \ldots + a_r(x)(t-c(x))^r,\]
with $a_i(x) \in QFD_*^k$; the $a_i(x)$ and $[t-c(x)]$ are bounded functions, and
there exists $n_0\in \N$ such that \[\ord f(x,t)-\min_i \ord [a_i(x)(t-c(x))^i] \leqslant n_0.\] 
\end{enumerate}
\end{proposition}
\begin{proposition}[Cell Decomposition-preparation II]\label{prop:cd2}
Let $t$ be one variable and $x=(x_1, \ldots, x_k)$. For $i=1,\ldots, r$, let $f_i(x,t)$ be  in $K_{QFD_*^k}[t]$. Let $n\in \N$ be fixed. There exists a finite partition of $K^k\times K$ into $\Lm_*$-cells $A$, such that on each cell $A$ (with center $c(x) \in QFD_*^k$), one of the following is true:
\begin{enumerate}
\item All functions $f_i(x,t)$ are unbounded on $A$ and for each $i$, \\either there exists $\lambda_i \in K, h_i(x) \in QFD_*^k$ such that
\[\ord f_i(x,t) =\ord h_i(x)\quad \text{and} \quad \rho_{n,m}(f_i(x,t)) = \rho_{n,m}( h_i(x)),\]
or there is $q_i\in K\backslash\{0\}$ and $\lambda_i \in K$ such that
\[\ord f_i(x,t) = \ord q_i(t-c(x)) \ \ \text{and} \ \ \rho_{n,m}(f_i(x,t)) = \rho_{n,m}( q_i(t-c(x))).\]
\item If $f_i(x,t)$ is unbounded on $A$, then there exist $\lambda_i \in K, h_i(x) \in QFD_*^k$ such that
\[\ord f_i(x,t) =\ord h_i(x)\quad \text{and} \quad \rho_{n,m}(f_i(x,t)) = \rho_{n,m}( h_i(x)).\]
If $f_i(x,t)$ is bounded on $A$, then there is a function $u_i(x,t)$ with $\ord u_i(x,t)=0$, ${\nu_i \in \N}$, $l' \in \Z$ and a bounded function $h_i(x) \in QFD_*^k$ such that \[f_i(x,t) = u_i(x,t)^n h_i(x)(t-c(x))^{\nu_i},\]
Moreover, also $t-c(x)$ is bounded on $A$.
\end{enumerate}
\end{proposition}
These propositions will be proven together, using induction on the degree of the polynomials. The first (resp. second) proposition for polynomials $f(x,t)$ of degree $\leqslant d$ will be refered to as $CDI_d$ (resp. $CDII_d$).\\\\
{\it Proof of $CDI_d$, assuming $CDI_{d-1}$ and $CDII_{d-1}$.}\\
Let $f(x,t) \in K_{QFD_*^k}[t]$ be an $\Lstar$-polynomial of degree $d$. Applying Proposition \ref{prop:lm*poly}, we get a partitioning of $K^{k+1}$ in (pre)cells $A$. Then either $f_i(x,t)$ is unbounded on $A$ and already has the form specified in the proposition, or $f(x,t)$ is bounded on A and can be written as
\[f(x,t) = a_0(x) +a_1(x)(t-c(x))+ \ldots + a_d(x)(t-c(x))^d,\]
using bounded functions $a_i(x) \in QFD_*^k$. We also know that $t-c(x)$ is bounded on $A$.
It remains to be checked that there exists some fixed upper bound for \[\ord f(x,t)-\min_i \ord [a_i(x)(t-c(x))^i]. \] For this we can proceed in exactly the same way as in the original proof for semi-algebraic sets, see Section 2.2 of \cite{denef-86}. We have to be careful about the definability of the used functions, every time when multiplication or division is used. Careful inspection shows that no problems occur; the only essential change is that we need to replace Lemma 2.3 and 2.4 of \cite{denef-86} by Lemma \ref{lemma:2.3bis} and  Lemma \ref{lemma:2.4bis}, respectively.

 \hfill{$\square$}\newline
\\
{\it Proof of $CDII_d$, assuming $CDI_{d}$.}\\
First look at the case $r=1$. We will check that there exists a partition in (pre)cells, such that  on each (pre)cell, $f(x,t)$ has one of the forms specified in the proposition. For precells on which $ f(x,t)$ is unbounded and has small enough order, this is clear.  On other precells, $f(x,t)$ is bounded, so that we can prove our claim in exactly the same way as in case 1 of Section 2.5 of \cite{denef-86} (again being careful about definability for multiplication and division).
\\\\
If $r>1$, repeat the above procedure for each $f_i(x,t)$. As we observed before, a precell can always be partitioned as a finite union of cells, and hence we get a partition of $K^{k+1}$ in cells with a center $c_0(x)$, such that on each cell there are `centers' $c_1(x), \ldots, c_r(x)$ and integers $l_1, \ldots, l_r$ such that each of the polynomials behaves in one of the following three ways 
\begin{enumerate}
\item[(a)] $\ord f_i(x,t)<l_i$ and there is $q_i\in \Q$ such that
\[\ord f_i(x,t) = \ord q_i(t-c_i(x)) \ \ \text{and} \ \ \rho_{n,m}(f_i(x,t)) = \rho_{n,m}(q_i(t-c_i(x))),\]
\item[(a')]$\ord f_i(x,t) <l_i$ and there exists $h_i(x) \in QFD_*^k$ such that
\[\ord f_i(x,t) =\ord h_i(x)\quad \text{and} \quad \rho_{n,m}(f_i(x,t)) = \rho_{n,m}(h_i(x)),\]
\item[(b)] there is a function $u_i(x,t)$ with $\ord u_i(x,t)=0, \nu_i \in \N$ and $h_i(x) \in QFD_*^k$ such that \[f_i(x,t) = u_i(x,t)^n h_i(x)(t-c_i(x))^{\nu_i},\]
with $\ord h_i(x)>l_i, \ord t-c_i(x) >l_i$.
\end{enumerate}
We only need to consider functions of types $(a)$ and $(b)$. 
Put $I = \{0,\ldots, r\}$. Let $I_b \subseteq I$ be the set of all $i \in I$ for which $t-c_i(x)$ is bounded on $C$.
Note that we may suppose  (possibly after a further partitioning) that for all $i\in I \backslash I_b$, $\ord (t-c_i(x)) <\min_{j \in I_b} \{\ord (t-c_j(x))\}$.

 If $I_b \neq \emptyset$, choose an element $i_b \in I_b$. Let $i_a$ be an element of $I\backslash I_b$. Then 
 $\ord (t-c_a(x)) = \ord (c_b(x)-c_a(x))$. By partitioning the cell $C$ further if necessary, we can assure that there exist fixed values $\lambda,\lambda'$ such that $t-c_a(x) \in \lambda P_{n}$ and $c_b(x) - c_a(x) \in \lambda' P_{n}$ for all $(x,t) \in C$. But this implies that for each $(x,t)$ there is a unique value $u^n(x,t)$ with $\ord u(x,t) = 0$, such that 
\[t-c_a(x) = \frac{\lambda}{\lambda'}u^n(x,t)(c_b(x)-c_a(x)).\]
The above formula allows us to express $\ord f_a(x,t)$ and $\rho_{n,m}(f_a(x,t))$ as the order, resp. residue of a function in $QFD_*^k$ in the variables $x$. 
In this way we can eliminate all centers $c_i(x)$ for $i \not \in I_b$. Now partition $C$ further in smaller cells, such that on each cell there exists some $i_0 \in I_b$ for which
\[\ord t-c_{i_0}(x) =\min_{i\in I_b}\{\ord(t-c_i(x))\}.\] By the same argument we used above we can (after further partitioning), eliminate all centers except $c_{i_0}$.

If $I_b =  \emptyset$, i.e. if all functions $f_i(x,t)$ are of types $(a)$ and $(a')$, we can eliminate superfluous centers in the same way as before. 
\hfill{$\square$}\newline
\begin{lemma} [Assuming $CDII_{d-1}$]\label{lemma:2.3bis}
Let $t$ be one variable and write $x = (x_1, \ldots, x_k)$. Let $g(x,t)$ be in $K_{QFD_*^k}[t]$ and assume that $g(x,t)$ can be written in the form
\(g(x,t) = \sum_{i=0}^d b_i(x)t^i,\)
where the coefficient functions $b_i(x)$ only take values from $R$. Let $e\in \N$, $\kappa \in R$ be fixed. If $\xi(x)$ is a function from $K^k \to R$ such that for all $x\in K^k$
\begin{align}
&g(x,\xi(x))=0\\
&\xi(x) \equiv \kappa \mod \pi^{e+1}\\
& 0\leqslant \ord g'(x,\xi(x))\leqslant e,
\end{align}
where g' denotes the derivative of $g$ with respect to $t$, then $\xi(x) \in QFD_*^k$. 
\end{lemma}
\begin{proof}
Let $S \subseteq K^r\times K$ be any quantifier-free definable set. We have to show that the set $A:= \{(y,x) \in K^r\times K^k \ | \ (y, \xi(x))\in S\}$ is also quantifier-free definable. A is a boolean combination of sets of relations of the forms
\begin{align}
&\ord f_1(y,\xi(x)) \ \square \ \ord f_2(y,\xi(x));\label{eq1:lemma2.3bis}\\
& f_3(y, \xi(x)) \in \lambda P_n,\label{eq2:lemma2.3bis}
\end{align}
where $f_i(y,t)$ is an $\Lm_*$-polynomial in $K_{QFD_*^r}[t]$.\\ 
What do we know about the $\Lm_*$-polynomials $f_i(y,t)$? 
Using Proposition \ref{prop:lm*poly}, we can find a partitioning of $K^{r+1}$, such that if $f_i(y,t)$ is unbounded on $C$, then  $\ord f_i(y,t) = \ord (q_it-c(y))$ and $\rho_{n,m}(f_i(y,t)) = \rho_{n,m}(q_it-c(y))$. In this case, we can replace $f_i(y,t)$ by a linear polynomial in $t$ when describing relations of the forms \eqref{eq1:lemma2.3bis} and \eqref{eq2:lemma2.3bis}.
\\
We may assume that given a condition \eqref{eq1:lemma2.3bis}, either both $f_1(y,t)$ and $f_2(y,t)$ are linear polynomials, or both are bounded and can thus be written in the form \eqref{eq8:lemma2.3bis}. If the $f_i(y,t)$ are bounded, apply Lemma \ref{lemma:ordtopn}, to obtain an $\Lm_*$-polynomial $\hat{f}(x,y,t)$ and $n\in \N$, such that (\ref{eq1:lemma2.3bis}) is equivalent to $\hat{f}(x,y,\xi(x)) \in P_n$.

Now assume that we have a condition of type \eqref{eq2:lemma2.3bis} and that $f_3(y,t)$ is bounded.
Note that  we can restrict our attention to those precells on which $\ord t \geqslant 0$. With these assumptions, the center $c(y)$ occurring in the expansion $f_3(y,t) = \sum c_{j}(x)(t-c(y))^i$ will also be a bounded function. 
Hence, there exists a partitioning (in precells)  such that $g(x,t)$ and $f_3(y,t)$  can be written as
\begin{equation}\label{eq8:lemma2.3bis} f_3(y,t) = \sum_{j=0}^{n_3} a_{j}(y)t^j\quad \text{and} \quad  g(x,t) = \sum_{i=0}^d b_i(x)t^i,\end{equation}
using bounded functions $a_{j}(y)$ and  $b_i(x)$. 
If  $f_3(y,t)$ has degree strictly bigger than $d$ in the variable $t$, we would like to use Euclidean division by $g(x,t)$ to reduce to an $\Lm_*$-polynomial of degree $\leqslant d-1$. However, at this point we cannot assure that the fractions $\frac{a_{j}(y)}{b_d(x)}$ will be definable. To remedy this, we can do the following. Put $N_d:= n_3-d+1$, and replace  $f_3(y,t)$ by $\hat{f}_3(x,y,t):= b_d(x)^{N_d}\cdot f_3(y,t)$. Since all coefficient functions occurring in this product are bounded, this product is definable, and moreover, after a further finite partitioning of $C$, the condition \eqref{eq2:lemma2.3bis} will be equivalent to a similar condition
\[ \hat{f}_3(x,y,t) \in \lambda'P_{n}.\]
(Note that we can assume that $b_d(x) \neq 0$ on $C$.) 
 We can now replace $\hat{f}_i(x,y,t)$ by its remainder after division by $g(x,t)$, thus obtaining a polynomial $\tilde{f}_3(x,y,t)= \sum_{j=0}^{d-1}a_{j}(x,y)t^j$ of degree at most $d-1$, whose coefficients $a_{j}(x,y) \in QFD_*^{k+r}$ are bounded functions.
Now apply $CDII_{d-1}$ to $\tilde{f}_3(x,y,t)$. We find that any relation of the form (\ref{eq2:lemma2.3bis}) is equivalent to (a finite number of) relations of the form \[qt-c(x,y) \in \lambda P_n,\] combined with a finite number of precell-conditions, and some quantifier-free definable relations in the variables $(x,y)$. Note that for any relation of this form, we may assume that $\ord t \geqslant 0$ and $c(x,y)$ is bounded. Indeed, if $\ord qt  > >\ord c(x,y)$, then the condition $qt-c(x,y) \in \lambda P_n$ is independent of the value of $t$.
%

Combining these observations, we have now reduced the problem to showing that any relation of the forms (remember that precells can be partitioned into $\Lm_*$-cells).
\begin{align*}
& \xi(x)-c(x,y) = \lambda P_n,\\
& \ord (\xi(x) -c(x,y)) \ \square \ \ord a(x,y),
\end{align*}
where $ c(x,y)$ and $ a(x,y)$ are bounded functions, definable without quantifiers. The proof of this case is exactly the same as the original proof of Lemma 2.3 of  \cite{denef-86}. (As always, one has to be careful with multiplication and division, but it is easy to see that no problems occur.)
\end{proof}
\begin{lemma} \label{lemma:ordtopn}Let $a(x,t), b(x,t) \in QFD_*^{k+1}$. Suppose that $a(x,t)$ and $b(x,t)$ are bounded, quantifierfree definable functions. There exists a function $c(x,t) \in QFD_*^{k+1}$  and $n \in \N$ such that for all $(x,t)\in K^{k+1}$,
 \[\ord a(x,t) \leqslant \ord b(x,t)\Leftrightarrow c(x,t) \in P_n.\]
 If $a(x,t), b(x,t)$ are in $K_{QFD_*^k}[t]$, then $c(x)$ is also in $K_{QFD_*^k}[t]$. \\Also, ${\ord c(x,t) \geqslant \min\{\ord a(x,t),\ord b(x,t)\}}$.
\end{lemma}
\begin{proof}
If $\ord 2 =0$, put $c(x,t) := a(x,t)^2 + \pi g(x,t)^2$. Then $\ord a(x,t) \leqslant \ord b(x,t)$ iff $c(x,t) \in P_2$. Since $a(x,t)$ and $b(x,t)$ are bounded, the function $c(x,t) \in QFD_*^k$. \\ If  $\ord 3 = 0$,  we can take $c(x,t):= a(x,t)^3 + \pi(b(x,t))^3$. Then $\ord a(x,t) \leqslant \ord b(x,t)$ iff $c(x,t) \in P_3$.   
\end{proof}
\begin{lemma}[Assuming $CDII_{d-1}$]\label{lemma:2.4bis}
Take  $\theta(x) \in QFD_*^r$. Suppose that $\theta(x) \neq 0$ and that there exists $l \in \Z$ such that $\ord \theta(x) \geqslant l$ for all $x \in K^r$. Let $k \in \N$, with $ 2 \leqslant k <d$. If for every $x\in K^r$, $\ord \theta(x)$ is a multiple of $k$, then there exists a function $\eta: K^r\to K$ in $QFD_*^r$, such that
\[\ord \eta(x) = \frac1k\ord \theta(x), \quad \text{ for all } x\in K^r.\]
\end{lemma}
\begin{proof}
The proof is essentially the same as that of the original Lemma 2.4 of  \cite{denef-86}. It is not hard to see that there exists a quantifier-free definable function $\theta_1(x) \in QFD_*^r$ such that $\ord \theta_1(x) = \ord \theta(x)$ and $\ac(\theta_1(x)) \equiv 1 \mod \pi^{2e+1}$, where $e = \ord k $. Using Hensel's lemma, it can be shown that for every $x\in K^r$, there exists a unique value $\eta(x)$ such that $\eta(x)^k = \theta_1(x)$ and $\ac\eta(x) \equiv 1 \mod \pi^{e+1}$. What we have to check is that $\eta(x)$ is in $QFD_*^r$.  By  arguments that are very similar to the ones we used in the proof of Lemma \ref{lemma:2.3bis} (with $g(x,t)$ replaced by the polynomial $t^k - \theta_1(x)$), we can show that it suffices to check that the following relations are qfd relations:
\begin{align*}
& \eta(x)-c(x,y) = \lambda P_n,\\
& \ord (\eta(x) -c(x,y)) \ \square \ \ord a(x,y),
\end{align*}
where $c(x,y)$ and $a(x,y)$ are bounded , quantifierfree definable functions. For this we can use exactly the same argument as was used in Lemma 2.4 of \cite{denef-86}.
\end{proof}

\subsection{Elimination of Quantifiers}
\begin{theorem}
Any subset of $K^{k+1}$ that is quantifier-free definable in $\Lstar$, can be partitioned as a finite union of $\Lm_*$-cells.
\end{theorem}
\begin{proof}
 First of all, we remark that the intersection of two $\Lm_*$-cells can again be written as a finite union of cells. For a related type of cells, a similar result was proved in \cite{clu-lee-2011}, Lemma 2.1. The proof given there can be applied (with some minor modifications) to $\Lstar$-cells. Therefore it suffices to check that the following sets can be partitioned as a finite union of cells:
\begin{align*}
D_1&:= \{(x,t)\in K^{k+1} \ | \ \ord f_1(x,t)\ < \ord  f_2(x,t)\},\\
D_2&:=\{(x,t)\in K^{k+1} \ | \ f_3(x,t) \in \lambda P_n \} ,
\end{align*}
where each $f_i(x,t)$ is an $\Lm_*$-polynomial in $K_{QFD_*^k}[t]$.
For $D_1$ we can do the following. Apply Proposition CDII to obtain a partition of $K^{k+1}$ in cells A. Then $D_1\cap A$ is equal to a set of one of the following forms 
\begin{align*}
I_1 &:= A \cap \{(x,t)\in K^{r+1} \ | \ \ord h_1(x)\ \square \ \ord h_2(x)\},\\
I_{1}' &:= A \cap \{(x.t) \in K^{r+1} \ | \ \ord q(t -c(x)) \ \square \ \ord  h_2(x) \},\\
I_2 &:= A \cap \{ (x,t) \in K^{r+1} \ | \ \ord a_1(x)(t-c(x))^{\nu_1} \ \square \ \ord a_2(x)(t-c(x))^{\nu_2}\},\\
I_3 &:=A \cap \{ (x,t) \in K^{r+1} \ | \ \ord a_1(x)(t-c(x))^{\nu_1} \ \square \ \ord h_2(x) \},
\end{align*}
with $h_i(x), a_i(x) \in QFD_*^k$, and $\square$ may denote either $<$ or $>$. Sets of types $I_1$ and $I_1'$ are already (intersections of $A$ with) cells. For  sets of type $I_2$ and $I_3$,  $a_i(x)$ and $t-c(x)$ are assumed to be bounded, say $\ord (t-c(x)) \geqslant l$. Suppose that $\nu_1 < \nu_2$. Note that  the
 condition \[\frac1{\nu_1-\nu_2}\ord \frac{a_1(x)}{a_2(x)} \ \square \ \ord (t-c(x))\] is either empty or trivial unless $\ord \frac{a_1(x)}{a_2(x)} \geqslant l(\nu_1-\nu_2)$, and hence we may assume that $\frac{a_1(x)}{a_2(x)}$ is definable.
It follows then from Lemma \ref{lemma:2.4bis} that a set of type $I_2$ is the intersection of $A$ with another cell, and thus can be partitioned as a finite union of cells. \\
Also for $I_3$ we have to be careful, since the function $\frac{h_2(x)}{a_1(x)}$ will probably not be definable. However, we can easily reduce this to the union of sets of type $I_1$ and $I_2$ using a further partitioning in sets $I_3 \cap \{(x,t) \in K^{k+1} \ | \ \ord h_2(x) \ \square \ \ord\pi^{(\nu_1+1)l}\}$.\\We can partition $D_2$ using a similar argument.
\end{proof}

\begin{theorem} \label{thm:celdec} Let $K$ be any $p$-adically closed field. Any $\Lm_*$-definable subset of $K^k$ can be partitioned as a finite union of $\Lm_*$-cells.
\end{theorem}
\begin{proof}
Let $C \subseteq K^{k+1}$ be the $\Lm_*$-cell 
\[C = \{(x,t) \in D \times K \ | \ \ord a_1(x) <
\ord(t-c) < \ord a_2(x), \ t-c \in \mu
P_n\},\] with $a_i(x) \in QFD_*^k$ and $D$
a definable subset of $K^{k}$.
Because of the previous theorem, it is sufficient to show that the following set can be defined without quantifiers:
\[\{(x,t) \in K^{k} \ | \ \exists t \in K: (x,t) \in C \}.\]
Note that 
 $A$ is in fact equal to
the following set:
\[ A = \{ x \in D \ | \ \exists \gamma \in \Gamma_K: \ord a_1(x) < \gamma< \ord
a_2(x), \ \gamma \equiv \ord \mu \mod n\}.\] Thus
$A$ is the set of all $x \in D$ satisfying
\begin{equation}\exists \gamma \in \Gamma_K: \frac{\ord
a_1(x)\mu^{-1}}{n} < \gamma <
\frac{\ord a_2(x)\mu^{-1}}{n}.\label{condforz}\end{equation}
Now if $\ord a_1(x)\mu^{-1} \equiv \zeta
\mod n$, for $0 \leqslant \zeta <n$, and , then condition
(\ref{condforz}) is equivalent with
\[ \ord a_1(x)\mu^{-1} + n - \zeta <
\ord a_2(x)\mu^{-1},\]which can be
simplified to
\[ \ord a_1(x) + n - \zeta <
\ord a_2(x).\]
 This completes
the proof, since the condition $\ord a_1(x)\mu^{-1}
\equiv \zeta \mod n$ is quantifierfree definable. \end{proof}

\section{Characterizations of definable functions}\label{sec:deffun}
In this section, we give some characterizations of  $\Lm_{*}$-definable functions. As a corollary, we will conclude that multiplication is only definable on bounded sets.
\begin{lemma}\label{lemma:linearpart}
For every $f(x) \in K_K[x]$ there exists a bounded $B(x) \in K_K[x]$ and $a_i, b \in K$ such that  
\[ f(x) = \sum_{i=1}^r a_ix_i +b  + B(x).\]
\end{lemma}
\begin{proof}
If $r=1$, this is an immediate consequence of Lemma \ref{lemma:lm*poly}. If $r >1$, there exists a function $\tilde{c} \in K_K[x_{r-1}, \ldots, x_1]$ and a bounded function $\tilde{B}(x)$
\[c(x_r; x_{r-1},\ldots,  x_1) = qx_r + \tilde{c}(x_{r-1}, \ldots, x_1) + \tilde B(x)\] 
Now apply induction on the number of variables $r$.
\end{proof}
\noindent We recall the following theorem by Denef.
\begin{theorem}[Denef, \cite{denef-84}]\label{thm:denef-polyorder}
Let $S \subseteq \Q_p^{m+1}$ be a semi-algebraic set. Suppose that for every $x\in \Q_p^m$, the set $\{\ord t \mid (x,t)\in S\}$ consists of exactly one element which we will denote by $\theta(x)$. Then there exists a finite partition of $\Q_p^m$ into semi-algebraic sets $A$, such that on each $A$ there are $e \in \N;$ $f_i(x)\in \Q_p[x]$ such that for all $x\in A$, \[\theta(x)= \frac1e\ord\frac{f_1(x)}{f_2(x)} .\]
\end{theorem}
\noindent Using this, we can show that
\begin{proposition}\label{prop:ordfunction1}
Let $f: K^r\to K$ be an $\Lm_{*}$-definable function. There exists a finite partitioning of $K^r$ in cells $C$, such that for every $C$, one of the following is true:
\begin{enumerate}
\item $f(x)$ is bounded on $C$ and there exist $a_i(x) \in K_K[x]$ and $e\in \Z$, such that 
\[\ord f(x) = \frac1e\ord\frac{a_1(x)}{a_2(x)},\]
with $\ord a_i(x) \geqslant 0$ for all $x\in C$,\\
\item   $f(x)$ is unbounded on $C$ and there exist $a_i, b \in K$ such that 
\[\ord f(x) = \ord \left(\sum_{i=1}^r a_ix_i +b\right).\] \end{enumerate}
\end{proposition}
\begin{proof}
The graph of $f(x)$ can be defined by a boolean combination of conditions 
\begin{align} &\ord f_1(x,t) <  \ord f_2(x,t), \quad \text{and} \label{eq1:orddeffun}\\ & f_3(x,t) \in \lambda P_n,\label{eq2:orddeffun}\end{align}
using functions $f_i(x,t) \in K_K[x,t]$. 

By Lemma \ref{lemma:linearpart}, we can partition $K^{r+1}$ in smaller sets $A$ such that each of the functions $f_i(x)$ is either bounded on $A$ or, if it is unbounded, its order is given by a linear polynomial and it  satisfies some condition $\ord f_i(x,t)<k$. If we make $k$ small enough, a condition of type \eqref{eq1:orddeffun} will be either trivial or false if  the $f_i(x,t)$ are not both (un)bounded on  $A$. If both $f_1(x,t)$ and $f_2(x,t)$ are bounded, apply lemma \ref{lemma:ordtopn} to replace the condition of type \eqref{eq1:orddeffun} by a condition of type \eqref{eq2:orddeffun}.
%

If  both functions are unbounded on $A$,
(\ref{eq1:orddeffun}) reduces to a relation between the valuations of linear polynomials. 
By a further cell decomposition we can either reduce such a relation to a condition independent of $t$, or we get a condition of the form
\[\ord (t-c(x))\ \square\  \ord b(x),\] with $b(x) \in K_K[x]$, $\square$ denotes either `$<$' or `$>$', and $c(x) = \sum a_i x_i +b$. This is independent of $t$ when $\ord c(x) <\ord t$. If $\ord c(x) \geqslant \ord t$, we conclude that either $\ord t= \ord c(x)$, which would prove our claim, or $\ord t \ \square\ \ord b(x)$.  
\\\\
We have now (up to partition) reduced the description of the graph of $f(x)$ to a number of conditions of the following forms:
definable conditions involving only $x$, not $t$, conditions of the form $\ord t \ \square \ \ord b_i(x)$ (here $\square $ may be assumed to denote either `$<$' or `$>$'), and conditions of the form $f_i(x,t) \in \lambda_i P_n$. 

Since $f$ is a function, for each fixed value of $x$, there must be a unique value $t = f(x)$ satisfying all of these conditions. However, conditions like $\ord t \ \square\ \ord \pi^kb_i(x)$ will at best fix $\ord t$. So for the cases we are considering here, we will always need at least one relation of type $f_i(x,t) \in \lambda_i P_n$ to define the set  $\text{Graph}(f)$.\\
First assume that every conditon of type $f_i(x,t) \in \lambda_i P_n$ is of the form $q(t-c_i(x)) \in \lambda_iP_n$,  with $c_i(x) \in K_K[x]$. For one of these conditions we should have that that $\lambda_i = 0$ (otherwise there would not exist a unique solution for $t$ for any given $x$). But then $f(x) = t = c_i(x)$, so that we can use Lemma \ref{lemma:linearpart}. 

If $t$ does not satisfy a linear conditon $q(t-c_i(x)) = 0$, at least one of the $f_i(x,t)$ has degree higher than one, and by Lemma \ref{lemma:lm*poly} we can assume that this $f_i(x,t)$ is bounded (possibly after a further partitioning), and $t$ satisfies a condition of the form 
\begin{equation} \sum_j a_j(x)(t-c_i(x))^j \in \lambda_i P_n\label{eq4:orddeffun}\end{equation}
with $a_j(x), t-c_i(x)$ bounded functions and $a_i(x), c_i(x) \in K_K[x]$.

Now if $t-c_i(x)$ is bounded, then if $\ord t$ is too small, we find that $\ord f(x) = \ord t = \ord c_i(x)$, and we can use Lemma \ref{lemma:linearpart}. Otherwise, if both $t$ and $t-c_i(x)$ are bounded, this implies that also $c_i(x)$ will be bounded, and  the condition $\sum_ja_j(x)(t-c_i(x))^j \in \lambda P_n$ can be rewritten to 
\begin{equation} \sum_ib_i(x)t^i \in \lambda P_n,\label{eq3:orddeffun}\end{equation}
with $b_i(x) \in K_K[x]$ bounded $\Lstar$-polynomials. Since we are now assuming that $\ord t$ is bounded,  linear conditions $q(t-c(x)) \in \lambda P_n$, will either be independent of $t$, or reduce to a condition of type (\ref{eq3:orddeffun}). 

We have now reduced to the case where $\ord t$ is bounded, and we have a  number of bounded conditions of type \ref{eq4:orddeffun}, together with additional conditions of the form (\ref{eq3:orddeffun}), and possibly some conditions $\ord t\ \square\ \pi^kb(x)$, so that we can now apply the same proof as for  Denef's Theorem \ref{thm:denef-polyorder}.  
(As always, we have to be a little careful about division, but it is easy to see that no problems occur.)
\end{proof}
We can now compare with $(K, \Laff)$, the structure of semi-affine sets.
It follows from the Proposition below that any $\Lstar$-definable function can be written as the sum of a bounded function and a function that is `essentially' semi-affine (note that the parts $X_i$ of the partition do not have to be semi-affine sets).
\begin{corol}\label{corol:fun}
Let $f: X \subseteq K^n \to K$ be an $\Lstar$-definable function. There exists a partition of $X$ into sets $X_i$, such that on each $X_i$, there is a linear polynomial $p_i(x) \in K[x]$ and a bounded $\Lstar$-definable function $b_i$   such that  $f _{|X_i}= p_i + b_i$.
\end{corol}
\begin{proof}
It follows from the proof of Proposition \ref{prop:ordfunction1} that there exist sets $U_1, \ldots U_r \subseteq K^n$ such that on each of the $U_i$, the function $f(x)$ is given by a polynomial in $K_K[x]$, and $f(x)$ is bounded on $X \backslash \bigcup_iU_i$. Now use Lemma \ref{lemma:linearpart}. 
\end{proof}
For functions in one variable we obtain a stronger result:
\begin{lemma}\label{lemma:funaffine}
Let $x$ be one variable and $f(x): X\subseteq K \to K$ an $\Lstar$-definable functon. There exists a bounded semi-affine set $B \subseteq K$ such that $f(x)$ is semi-affine on $X \backslash B$.
\end{lemma}
\begin{proof}
By Proposition \ref{prop:ordfunction1}, $x$ is bounded whenever $f(x)$ is bounded. Therefore, the set $X \backslash \bigcup_iU_i$ we obtained in the proof of the previous lemma must be bounded, and this set is semi-affine since it is a definable subset of $K$. Moreover, on each of the (semi-affine) sets $U_i$, the function $f(x)$ is given by some polynomial in $K_K[x]$.

Applying Lemma \ref{lemma:lm*poly}, we find a partitioning of $U_i$ in sets $D$, such that on each $D$, we have that $f(x) = \sum_{j=1}^r f_j(x)$, where
\[ f_j(x) = \prod_{i=1}^{r_j} s_{ij}\left[a_{ij}*(q_{ij}x + b_{ij})\right]\cdot \prod_{i=1}^{r_j'}(v_{ij}x + c_{ij}),\]
for given constants $s_{ij}, a_{ij}, q_{ij}, b_{ij}, v_{ij}, c_{ij}$. 
We also know that the functions $v_{ij}x + c_{ij}$ are bounded on $D$, which implies that $x$ has to be bounded on $D$ (unless $(r_j,r_{j}' )= (0,1)$ or $r_{j}' =0$).  So if $D$ is not a bounded set, then for each $j$, either  $r_j' =0$ or $f_j(x)$ is a linear polynomial. 
If $D$ is unbounded, put  \[Y_D := \left\{ x \in D \mid \ord x < \min_{i,j}\{- \ord q_{ij}, \ord \frac{b_{ij}}{q_{ij}}\}\right\}.\] Then $X \backslash \cup Y_D$ is a bounded set, and on each $Y_D$, $f(x)$ is  a linear polynomial.  
\end{proof}

\begin{corol}\label{corol:multdef}
If  $f: X\subseteq K^2 \mapsto K: (x,y) \mapsto xy$ is $\Lm_{*}$-definable, then $X$ is a bounded set.
\end{corol}
\begin{proof}
By the  commutativity of the multiplication map, any maximal set $X\subseteq K^2$ on which multiplication is definable, must be of the
 form $X = Y \times Y$, for some $Y \subseteq K$. Moreover, if multiplication is definable on $X$,  this induces an $\Lstar$-definable function $g:Y\to K : x\mapsto x^2$. By Lemma \ref{lemma:funaffine}, there is a bounded semi-affine set $B$ such that $g$ is semi-affine on $Y \backslash B$. Hence, 
%
after partitioning $Y \backslash B$ into a finite number of smaller sets $S$, there exist constants $a_S\in \Q, \ b_S\in K$ such that on each $S$,
\[ g_{|S}(x) = a_Sx+b_S.\]
Since all elements of $S$ should satisfy the equation $x^2 = a_Sx + b_S$, each $S$ is a finite set. So $Y$ must be a bounded set, and therefore the domain of $f$ will also be bounded. 
\end{proof}

\subsection{On the existence of definable Skolem functions.}
In \cite{lee-2011b}, we show the following link between cell decompositions that use strong cells and the existence of definable Skolem functions. Here $\Ldist$ is the language $\LmM \cup \{ \ord(x-y)<\ord(z-t)\}$, and $\overline{\Q}^K$ is the algebraic closure of $\Q$ in $K$.
\begin{theorem}
Let $K$ be a $p$-adically closed field. Suppose that $\Lm \supseteq \Ldist$ and that multiplication by constants from $\overline{\Q}^{K}$ is definable in $(K, \Lm)$. The structure $(K,\Lm)$ admits strong cell decomposition
if and only if $(K, \Lm)$ has definable Skolem functions.
\end{theorem}
It is then immediate that $(K, \Lstar)$ has definable Skolem functions for all $p$-adically closed fields $K$. To achieve cell decomposition for structures $(K, \Lstar^F)$ with $F \subset K$, we need to modify the definition of cells to
\[\{(x,t) \in D \times K \mid \ord \pi^{k_1}a_1(x) \ \square_1\  \ord (t-c(x))\  \square_2 \ \ord \pi^{k_2}a_2(x), \ t-c(x) \in \lambda P_{n}\},\] with $k_i \in \Z$ and $a_i(x)$ a quantifierfree $\Lstar^F$-definable function. This decomposition is strong if there exist $\Lstar^F$-definable functions with the same order as $\pi^{k_i}a_i(x)$. Hence such a structure will have definable Skolem functions if
multiplication by constants from $\overline{\Q}^K$ is definable in $(K, \Lstar^F)$, since $\overline{\Q}^K$ contains an element $\pi$ of minimal positive order (we refer to \cite{lee-2011b} for a proof of this fact). 

%

\section{Semi-bounded sets: comparison with the real situation}\label{sec:compare}

For real closed fields, it has been established  that there is a unique substructure of  the semi-algebraic sets that properly extends the semi-linear sets: the structure of  semi-bounded sets.
In the next theorem, we show that $(K, \Lstar)$ is the natural $p$-adic equivalent to this structure.

\begin{theorem}\label{prop:padicsemibounded} Let $X \subseteq K^n$ be a semi-algebraic set. The following conditions are equivalent.
\begin{enumerate}
\item $X$ is a semi-bounded set: there exist bounded semi-algebraic sets $X_1, \ldots, X_r$ such that $X$ is definable in $\Laff \cup\{X_1, \ldots, X_r\}$.
\item $X$ is $\Lstar$-definable
\end{enumerate}
\end{theorem}
\begin{proof}
 Any bounded semi-algebraic set $B$ is $\Lstar$-definable: write $B_{\gamma}$ for the set $\{x \in K^n \mid \min \ord x_i \geqslant \gamma\}$, and choose  $\gamma$   such that $B \subseteq B_{\gamma}$. Addition and multiplication are $\Lstar$-definable on $B_{\gamma}$, and hence there exists an $\Lstar$-formula $\phi(x)$, such that $\{x \in B_{\gamma} \mid \phi(x)\} = B$. It follows immediately that every semi-bounded set is also $\Lstar$-definable. 
On the other hand, every $\Lstar$-definable set is semi-bounded since the graph of $*$ is definable in $\Laff \cup X_1$, where $X_1$ is the bounded semi-algebraic set \[X_1 = \{(x,y,z) \in K^3 \mid \ord x = \ord y=0 \wedge z = xy\}.\]
\end{proof}
Pillay, Scowcroft and Steinhorn \cite{pil-sc-st-89} showed that if $X \subseteq \R^n$ is a bounded set, then multiplication is not definable in $(\R, + , <, \{\overline{c}\}_{c\in \R}, X)$. 
We have a similar result:
\begin{theorem}
 Multiplication is not  definable in $\Laff \cup\{X\}$ if $X\subset \Q_p^n$ is a bounded set
\end{theorem}
\begin{proof}
Such a structure would be a substructure of the semi-bounded sets $(\Q_p, \Lstar)$, so our claim follows by Corollary \ref{corol:multdef}.
\end{proof}
%
If $X$ is semi-bounded, this also imposes the following restriction on the definable sets of $\Laff \cup \{X\}$:

\begin{lemma} \label{lemma:2dimsets} If $X \subseteq K^n$ be a semi-bounded set,
 then every subset of $K^2$ that is definable in $\Laff \cup \{X\}$ is semi-affine outside some bounded set
\end{lemma}
\begin{proof}
We need to show that any $\Lstar$-definable set $A \subseteq K^2$ is semi-affine outside some bounded set. Without loss of generality, we may assume that $A$ is a cell
\[A:= \{(x,t) \in D \times K \mid \ord a_1(x) \ \square_1 \ \ord(t-c(x)) \ \square_2 \ \ord a_2(x), \ t-c(x) \in \lambda P_n\}.\]
By Corollary \ref{lemma:funaffine}, we can find a bounded (semi-affine) set $B$ such that for all $x \in D \backslash B$, the functions $a_i(x)$ and $c(x)$ occurring in the description of the cell are all semi-affine. Hence, it is clear that $A \cap (D \backslash B \times K)$ is a semi-affine set. 

If  $x$ is bounded, $a_i(x)$ and $c(x)$ will also be bounded, so $A_2:= A \cap (B\times K)$ is clearly a bounded set if $\square_1$ denotes `$<$'. 
Otherwise, if $\square_1$ denotes `no condition', choose a constant $a \in K$ such that $a < \min_{x\in B}\{ \ord c(x) -m, \ord a_2(x)\}$. Now partition $A_2$ as $E \cup A_2 \backslash E$, where \[ E:= \{(x,t) \in A_2 \mid \ord (t-c(x)) \geqslant a\}.\]
Clearly, $E$ is a bounded set. Note that on $A_2\backslash E$,  we have that $\ord t = \ord(t-c(x))$ and $\rho_{n,m}(t-c(x)) = \rho_{n,m}(t)$. Hence, the description of this set simplifies to
\[ A_2 \backslash E = \{(x,t) \in B \times K \mid \ord t < a \wedge t \in \lambda Q_{n,m}\},\] which is a semi-affine set.
\end{proof}


Peterzil \cite{pet-92,pet-93} 
obtained that $(\R, +, < \cdot_{|[-1,1]}, \overline{c}_{c\in \R})$) is the only structure between the semi-linear and semi-algebraic sets. He used a number of equivalent characterizations of real semibounded semi-algebraic sets, which included
\begin{theorem}[Peterzil] \label{thm:realsemibounded} Let $X \subseteq \R^n$ be a semi-algebraic set. The following are equivalent
\begin{enumerate}
\item $X$ is a semi-bounded set
\item In $(\R, + , \{\overline{c}\}_{c\in \R}, <,X)$, one cannot define a bijection between a bounded and an unbounded interval
\item Every curve in $\R^2$ that is definable in $(\R, + , \{\overline{c}\}_{c\in \R}, <,X)$, is semilinear outside some bounded subset of $R^2$
\end{enumerate}
\end{theorem} 
Note that, in contrast with Peterzil's result, 
the equivalence $(1) \Leftrightarrow (3)$ does not hold in the $p$-adic context.
This can be seen from the following example.
Consider the structure $(K, \Laff \,\cup\, M_2)$, with \[M_2:= \{(x,y) \in K \mid \ord y = \ord x^2\}.\] We can show (see the preprint \cite{lee-2012.1} on my webpage for the details) that this structure essentially has the same definable functions as $\Laff$, in the sense that, up to a finite partioning of the domain, every function is given by linear polynomials, and therefore any curve in $K^2$ will be semi-affine outside some bounded set.
Yet the set $M_2$ is not semi-bounded.
Instead, we conjecture that the converse of Lemma \ref{lemma:2dimsets} holds for structures on $p$-adically closed fields. 

In the $p$-adic context, it is also false that a set is semi-bounded if and only if there does not exist a bijection between a bounded and an unbounded ball in $\Laff \cup X$. Clearly, such a bijection cannot exist if $X$ is a semi-bounded set, as any function $\Lstar$-definable function $f:K\to K$ is semiaffine outside some bounded interval. 
%
%
However,  $(K, \Laff \cup M_2)$ is an example of a structure that contains sets that are not semibounded, and yet does not allow such a definable bijection.
%
We will disscuss such structures in more detail in a next paper \cite{lee-2012.1}.

\bibliographystyle{abbrv}

\bibliography{/Users/iblueberry/Documents/Bibliografie}
\end{document}